\documentclass{lms}

\usepackage[breaklinks=true]{hyperref}
\usepackage{amsmath,amssymb, color}
\newtheorem{theorem}{Theorem}[section]

\newtheorem{lemma}[theorem]{Lemma}
\newtheorem{prop}[theorem]{Proposition}
\newtheorem{cor}[theorem]{Corollary}
\newtheorem{utheorem}{\textrm{\textbf{Theorem}}}

\newcommand{\bd}[1]{d^{[#1]}}

\newnumbered{defn}[theorem]{Definition}
\newnumbered{rem}[theorem]{Remark}
\newnumbered{example}[theorem]{Example}

\providecommand{\adj}{\mathop{\rm adj}}
\providecommand{\omg}{\mathop{\rm \Omega}}
\providecommand{\sgn}{\mathop{\rm sgn}}

\newcommand{\altr}[1]{\mathbb{R}^{#1}_{\rm alt}}
\newcommand{\lcp}[1]{\mathrm{LCP}(#1,q)}
\newcommand{\lcpn}[1]{\mathrm{LCP}(#1)}
\newcommand{\sol}[1]{\mathrm{SOL}(#1,q)}
\newcommand{\rn}[1]{\mathbb{R}^{#1}}

\title[Characterizing total positivity: single vector tests]
 {Characterizing total positivity: single vector tests via Linear Complementarity, sign non-reversal, and variation diminution} 

\author{Projesh Nath Choudhury}

\classno{15B48, 90C33 (primary), 15A24 (secondary)}

\extraline{This work is supported by National Post-Doctoral
	Fellowship(PDF/2019/000275) from SERB (Govt.~of India)}

\begin{document}
\maketitle

\tableofcontents

\begin{abstract}
A matrix $A$ is called totally positive (or totally non-negative) of order $k$,
denoted by $TP_k$ (or $TN_k$), if all minors of size at most $k$ are
positive (or  non-negative). These matrices have featured in diverse areas in mathematics, including algebra, analysis, combinatorics, differential equations and probability theory. The goal of this article is to provide a novel connection between total positivity and optimization/game theory. Specifically, we draw a relationship between totally positive matrices and the Linear Complementarity Problem (LCP), which generalizes and unifies linear and quadratic programming problems and bimatrix games -- this connection is unexplored, to the best of our knowledge. We show that
$A$ is $TP_k$ if and only if for every submatrix $A_r$ of $A$ formed from $r$ consecutive rows and $r$ consecutive columns (with $r\leq
k$), $\lcp{A_r}$ has a unique solution for each vector $q<0$. In fact this can be
strengthened to check the solution set of the LCP at a single vector for each such square submatrix. These novel characterizations are in the spirit of classical results characterizing $TP$ matrices by Gantmacher--Krein
[\textit{Compos.\ Math.} 1937] and $P$-matrices by Ingleton
[\textit{Proc. London Math. Soc.} 1966].

Our work contains two other contributions, both of which characterize
total positivity using single test vectors whose coordinates have alternating signs -- i.e., lie in a certain open bi-orthant. First, we improve on one of the main results in recent joint work [\textit{Bull.\ London Math.\ Soc.}, 2021], which provided a novel characterization of $TP_k$ matrices using sign non-reversal phenomena. We further improve on a classical characterization of total positivity  by Brown--Johnstone--MacGibbon [\textit{J. Amer. Statist. Assoc.} 1981] (following Gantmacher--Krein, 1950) involving the variation diminishing property. Finally, we use a P\'olya frequency function of Karlin [\textit{Trans.\ Amer.\ Math.\ Soc.} 1964] to show that our aforementioned characterizations of total positivity, involving (single) test-vectors drawn from the `alternating' bi-orthant, do not work if these vectors are drawn from any other open orthant.
\end{abstract}

\section{Introduction and main results}

Given an integer $k\geq 1$, we say a matrix is \textit{totally positive of order
	$k$ ($TP_k$)} if all its minors of order at most $k$ are positive. A matrix $A$ is
\textit{totally positive (TP)} if $A$ is $TP_k$ for all $k\geq 1$ i.e,  all minors of $A$ are positive. Similarly,
one defines \textit{totally non-negative ($TN$)} and $TN_k$ matrices for $k\geq 1$. These classes of matrices have important applications in various theoretical and applied
branches in mathematics. We mention a few of these topics and some of the experts who worked on them: analysis (Fekete and P\'olya \cite{FP12}, Schoenberg \cite{S30,S55}, Whitney \cite{Whitney}), representation theory (Lusztig \cite{Lu94}, Rietsch \cite{Ri03}), cluster
algebras (Berenstein, Fomin and Zelevinsky \cite{BFZ96,FZ02}), combinatorics (Brenti \cite{Bre95}),  matrix theory (Fallat and Johnson \cite{fallat-john}, Garloff \cite{GW96}, Pinkus \cite{pinkus}), differential equations (Karlin \cite{K68}, Loewner \cite{Lo55}), Gabor analysis (Gr\"{o}chenig, Romero and St\"{o}ckler \cite{GRS18}),
integrable systems (Kodama and Williams \cite{KW14}), probability and statistics (Karlin \cite{K68}), interacting particle systems (Gantmacher and Krein \cite{gantmacher-krein,GK50}), and interpolation theory and splines (de Boor \cite{deBoor}, Karlin and Ziegler \cite{Karlinsplines}, and Schoenberg with collaborators \cite{Curry,Schoenberg46,SW53}). We also mention the preprints \cite{BGKP20,BGKP21} for preserver problems involving totally positive matrices and P\'olya frequency functions/sequences. Given these numerous strong connections to many subfields of the broader mathematical sciences, it is perhaps surprising that a characterization of total positivity in terms of optimization/game theory or an application of total positivity in optimization/game theory remain  unexplored, to the best of our knowledge. The main objective of this article is to draw a connection between total positivity and the Linear Complementarity Problem (LCP) which generalizes and unifies linear and quadratic programming problems and bimatrix games. We believe this connection is novel.

Given a matrix $A \in \mathbb{R}^{n \times n}$ and a vector $q \in \mathbb{R}^n$, the {\it Linear Complementarity Problem} $\lcp{A}$ asks to find, if possible, a vector $x \in \mathbb{R}^n$ such that
\begin{equation}
x \geq 0,\quad y=Ax+q \geq 0, \quad and \quad x^Ty=0,\label{lcpdefn}
\end{equation}
where $x=(x_1,\ldots,x_n)^T\geq 0$ means that all $x_i\geq 0$. Any $x$ that satisfies the first two conditions is called a \textit{feasible vector}, and any feasible vector that satisfies the third condition is called a \textit{(complementarity) solution} of $\lcp{A}$. The set of all solutions of $\lcp{A}$ is denoted by $\sol{A}$.
The Linear Complementarity Problem has important applications in many different areas, including bimatrix games, convex quadratic programming, fluid mechanics, solution of systems of piecewise linear
equations, and variational inequality problems \cite{CD68,CPS09,Cr71,ES76,L65}. The study of LCPs has resulted in progress along several fronts. For example, the complementary pivot algorithm, which was first developed for solving LCPs, has been generalized in a straightforward way to obtain efficient algorithms for computing Brouwer and Kakutani fixed points, for solving systems of nonlinear equations and nonlinear programming problems, and for computing economic equilibria. Also, iterative methods developed for solving LCPs are very useful for tackling very large scale linear programs, which cannot be handled with the simplex method because of their large size and numerical difficulties. As far as the bimatrix game is concerned, the LCP formulation was instrumental in the discovery of an efficient constructive method for the computation of a Nash equilibrium point. For more details about Linear Complementarity Problems and their applications, we refer to \cite{Co68,CPS09,I70}.

In this section, we state our first two main results, which provide characterizations of total positivity in terms of the Linear Complementarity Problem, and thereby connect these two well-studied areas. To state these results, we begin with some preliminary definitions, which we use in this paper without further reference.

\begin{defn}
	Let $n \geq 1$ be an integer, $A \in \mathbb{R}^{n \times n}$ be a matrix, and $S \subset \mathbb{R}^n$ be a subset.
	\begin{enumerate}
		\item Define the set $[n]:=\{1,2,\ldots,n\}$.
		\item Given $i,j \in [n]$, let $A^{ij}$ denote the determinant of the $(n-1)\times (n-1)$ submatrix of $A$ obtained by deleting the $ith$ row and $jth$ column of $A$. If $n=1$, then define $A^{11}:=1$ to be the determinant of the empty matrix.
		\item Let $\adj(A)$ denote the adjugate matrix of $A$ and $\{e^i\}$ denote the standard orthonormal basis of $\mathbb{R}^n$.
		\item Let $\altr{n} \subset \mathbb{R}^n$ denote the set of real vectors with all nonzero coordinates and alternating signs.
		\item A submatrix $B$ of $A$ is called a \textit{contiguous submatrix}, if the rows and columns of $B$ are indexed by sets of consecutive integers.
		\item The matrix $A$ has the \textit{sign
			non-reversal property} with respect to $S$, if for all vectors $0 \neq x
		\in S$, there exists a coordinate $i \in [n]$ such that $x_i (Ax)_i
		> 0$.
		
		\item We also require a non-strict version. The matrix $A$ has the \textit{non-strict sign non-reversal
			property} with respect to $S$ if for all vectors $0 \neq x \in S$, there exists a coordinate $i \in [n]$ such that $x_i \neq 0$ and $x_i
		(Ax)_i \geq 0$.
		\item We say that a vector $x \in \mathbb{R}^n$ is $\geq 0$ (respectively $x>0,~ x\leq 0,~ x<0$) if every coordinate of $x$ is $\geq 0$ (respectively $>0,~\leq 0,~<0$).
	\end{enumerate}
\end{defn}
We can now state our first main result. 
\begin{utheorem}\label{tp-lcp}
	Let $m,n \geq k \geq 1$ be integers.
	Given $A \in \mathbb{R}^{m \times n}$, the following statements
	are equivalent.
	\begin{enumerate}
		\item The matrix $A$ is totally positive of order $k$.
		\item For every square submatrix $A_r$ of $A$ of size $r\in [k]$, $\lcp{A_r}$ has a unique solution for all $q \in \rn{r}$.
		\item For every contiguous square submatrix $A_r$ of $A$ of size $r\in [k]$, $\lcp{A_r}$ has a unique solution for all $q \in \rn{r}$ with $q<0$.
		\item For every contiguous square submatrix $A_r$ of $A$ of size $2\leq r \leq k$ and for all $q \in \rn{r}$ with $q<0$, $\sol{A_r}$ does not simultaneously contain two vectors with sign pattern $\begin{pmatrix}
		+ \\ 0 \\+\\0\\\vdots
		\end{pmatrix}$ and $\begin{pmatrix}
		0 \\ + \\0\\+\\\vdots
		\end{pmatrix}$ . Moreover for $r=1$ and all $i\in [m],~j \in[n]$, $\lcp{(a_{ij})_{1\times 1}}$ has a solution for some scalar $q<0$. \smallskip
		
	\end{enumerate}
\end{utheorem}

In fact, we improve this result by characterizing total positivity in terms of the number of solutions of the LCP at a single vector, for each contiguous submatrix.

\begin{utheorem}\label{tp-lcp_1}
	Let $m,n \geq k \geq 1$ be integers.
	Given $A \in \mathbb{R}^{m \times n}$, the following statements
	are equivalent.
	\begin{enumerate}
		\item The matrix $A$ is totally positive of order $k$.
		\item For every $r \in [k]$ and contiguous $r
		\times r$ submatrix $A_r$ of $A$, define the vectors
		\begin{equation}\label{tlcp}
		x^{A_r} := (A^{11}_r,0,A^{13}_r,0,\ldots)^T, \qquad q^{A_r}: =-A_r x^{A_r}.
		\end{equation}
		Then $x^{A_r}$ is the only solution of $\lcpn{A_r,q^{A_r}}$.
		
	\end{enumerate}
\end{utheorem}
\noindent Here, the determinant of the empty matrix is defined to be one.

An immediate application of Theorem~\ref{tp-lcp}
is a novel characterization of \textit{P\'olya frequency sequences of
	order $k$} via linear complementarity. These are real sequences $(c_n)_{n \in \mathbb{Z}}$
such that for all integers
\[
1 \leq r \leq k, \qquad m_1 < \cdots < m_r, \qquad n_1 < \cdots < n_r,
\]
the determinant $\det (c_{m_i - n_j})_{i,j=1}^r \geq 0$. If all such
determinants are positive, we say the sequence is a
\textit{$TP_k$ P\'olya frequency sequence}.

\begin{cor}
	Let $k \geq 1$ be an integer. A real sequence $(c_n)_{n \in \mathbb{Z}}$
	is a $TP_k$ P\'olya frequency sequence, if and only if for all integers
	$r \in [k]$ and $l \in \mathbb{Z}$, and all vectors $q \in \mathbb{R}^{r}$ with $q<0$, there
	exists a unique $x \in {[0,\infty)}^r$ such that
	\begin{equation}
	y_i:=\sum_{j=1}^r c_{l + i -
		j} x_j +q_i\geq 0,\quad x_i y_i=0,\quad \forall i \in [r].
	\end{equation}
\end{cor}

This can be shown by applying Theorem~\ref{tp-lcp} to the
square submatrices
\[
\begin{pmatrix}
c_l & c_{l-1} & \cdots & c_{l-r+1}\\
c_{l+1} & c_l & \cdots & c_{l-r+2}\\
\vdots & \vdots & \ddots & \vdots\\
c_{l+r-1} & c_{l+r-2} & \cdots & c_l
\end{pmatrix}, \qquad l \in \mathbb{Z}.
\]

We now explain the organization of the paper, including two additional results below. In the next section we prove Theorems \ref{tp-lcp} and \ref{tp-lcp_1}. In addition to classical results by Fekete, Ingleton, Samelson--Thrall--Wesler, and Schoenberg, the key new ingredient is a novel characterization of total positivity (in fact of $TP_k$), which uses the sign non-reversal property and which was not known until our recent joint work \cite{CKK21}.

In Section \ref{snrsec}, we return to our recent joint work \cite{CKK21}, in which we showed that $TN_k$ matrices are characterized by the sign non-reversal property at a single vector (for each square submatrix). A similar result for $TP_k$ remained elusive, though we had shown a characterization involving the sign non-reversal property at a single vector, but with an (uncountable) additional set of constraints. Our next main result, Theorem \ref{TPsnr}, addresses this gap and provides truly a single test vector for each square submatrix of a $TP_k$ matrix without additional conditions.

In Section \ref{vdsec} we return to an even earlier, fundamental result of Gantmacher and Krein in 1950 \cite{GK50} (see also its stronger version in \cite{BJM81}). This is a well-known characterization of total positivity, in terms of the variation diminishing property on the test set of all real vectors,
which has numerous applications (see \cite{K68,MS19}). Our final main result, Theorem \ref{TP-vandimnew}, improves on this by reducing the test set to a single vector for each square submatrix.

In the final section, we take a second look at our results in the previous two sections \ref{snrsec} and \ref{vdsec}. We showed in these two sections (and previous work) that the variation diminishing property and the sign non-reversal property, each at a single test vector (for each square submatrix of $A$) suffices to prove the total positivity of the matrix $A$. The proofs reveal that these test vectors necessarily have coordinates with alternating signs. We now show that such `single test vectors' \textit{must} have alternating-signed coordinates. Namely, any $TP_{n-1}$ or $TN_{n-1}$ matrix in $\mathbb{R}^{n \times n}$, even one with a negative determinant, satisfies the variation diminishing property and the sign non-reversal property on every vector in every other (open) orthant in $\mathbb{R}^n$. We also provide a similar observation about the LCP.

We conclude this section with some general remarks. In 1937,
Gantmacher--Krein \cite{gantmacher-krein} gave a fundamental
characterization of totally positive matrices of order $k$ by the
positivity of the spectra of all submatrices of size at most $k$. There
is a well known article \cite{FZ00} by Fomin--Zelevinsky about tests for
totally positive matrices; there have been numerous subsequent papers along this theme, e.g. \cite{Chepuri}. The present paper may be regarded as being similar in spirit.

\section{Theorems \ref{tp-lcp} and \ref{tp-lcp_1}: Total positivity and the Linear Complementarity Problem}\label{tpsec}
In this section we prove Theorems \ref{tp-lcp} and \ref{tp-lcp_1}. To proceed, we require two preliminary results. The first result establishes a connection between the LCP and matrices with positive principal minors (these are known as $P$-matrices):
\begin{theorem}[(Ingleton \cite{I66}, Samelson--Thrall--Wesler~\cite{STW58})]\label{lcpp}
	A matrix $A \in \mathbb{R}^{n \times n}$ has all principal minors
	positive if and only if 
	$\lcp{A}$ has a unique solution for all $q \in \mathbb{R}^{n}$.
\end{theorem}

The proof of Theorem \ref{tp-lcp} also uses the following result, proved in recent joint work, which characterizes total positivity in terms of the sign non-reversal phenomenon.

\begin{theorem}\cite{CKK21}\label{tp-sign-rev_k}
	Let $m,n \geq k \geq 1$ be integers.
	Given $A \in \mathbb{R}^{m \times n}$, the following statements
	are equivalent.
	\begin{enumerate}
		\item The matrix $A$ is totally positive of order $k$.
		\item Every square submatrix of $A$ of size $r \in [k]$ has the
		sign non-reversal property with respect to $\mathbb{R}^r$.
		\item Every contiguous square submatrix of $A$ of size $ r \in [k]$ has the sign non-reversal property with respect to $\altr{r}$.
	\end{enumerate}
\end{theorem}

We will also revisit and strengthen this result in Theorem \ref{TPsnr} below.
\begin{proof}[of Theorem \ref{tp-lcp}]
	That $(i) \implies (ii)$ follows from Theorem \ref{lcpp}, while $(ii)
	\implies (iii) \implies (iv)$ is immediate. To show $(iv) \implies (i)$, we first claim that all entries of $A$ are positive. Indeed, if $a_{ij}\leq 0$ for some $i\in[m],~j \in [n]$, then $\lcp{(a_{ij})_{1\times 1}}$ does not have a solution for any scalar $q<0$.
	
	Next, we show that all the minors of $A$ of size $2\leq r \leq k$ are positive. By Theorem \ref{tp-sign-rev_k}, it suffices to show every contiguous square submatrix of $A$ of size $ r \in [k]\setminus\{1\}$ has the sign non-reversal property with respect to $\altr{r}$. Let $r \in [k]\setminus\{1\}$ and suppose for contradiction that $A_{r}$ is an $r \times r$ contiguous submatrix of $A$ such that $A_{r}$  does not satisfy the sign non-reversal property with respect to $\altr{r}$. Then there exists $x \in \altr{r}$ such that $x_i(A_r x)_i\leq 0$ for all $i \in [r]$. Let $x^+:=\frac{1}{2}(\vert x \vert +x)$ and $x^-:=\frac{1}{2}(\vert x \vert -x)$, where we define $\vert (x_1,\ldots,x_n)^T\vert:= (\vert x_1 \vert, \ldots, \vert x_n \vert)^T $. Then $x^+$ has sign pattern $\begin{pmatrix}
	+ \\ 0 \\+\\0\\\vdots
	\end{pmatrix}$ and  $x^-$ has sign pattern $\begin{pmatrix}
	0 \\ + \\0\\+\\\vdots
	\end{pmatrix}$ (or vice-versa). Let $v=A_rx$ and set $v^{\pm}:=\frac{1}{2}(\vert v\vert \pm v)$. Note that $x=x^+-x^-$ and $A_rx=v^+-v^-$. Define
	\begin{equation}
	q:=v^+-A_rx^+=v^--A_rx^-.\label{thrmaeq1}
	\end{equation}
	Since $x^+, x^-\geq 0$ and $A_r>0$, we have $q<0$. Also, ${(x^+)}^T v^+ =0$ and ${(x^-)}^T v^- =0$, since $x_i v_i\leq 0$ for all $i \in [r]$. Thus $\lcp{A_r}$ has solutions having sign patterns $\begin{pmatrix}
	+ \\ 0 \\+\\0\\\vdots
	\end{pmatrix}$ and $\begin{pmatrix}
	0 \\ + \\0\\+\\\vdots
	\end{pmatrix}$. This yields the desired contradiction.
\end{proof}
To prove Theorem \ref{tp-lcp_1}, we require the well-known 1912 result of Fekete for $TP$ matrices, which was extended in 1955 by Schoenberg to $TP_k$
matrices.
\begin{theorem}[(Fekete~\cite{FP12}, Schoenberg~\cite{S55})]\label{fec}
	Let $m,n \geq k \geq 1$ be integers. Then $A \in \mathbb{R}^{m \times n}$
	is $TP_k$ if and only if every contiguous square submatrix of $A$ of size $r \in [k]$ has positive determinant.
\end{theorem}

\begin{proof}[of Theorem \ref{tp-lcp_1}]
	That $(i) \implies (ii)$ follows from Theorem \ref{tp-lcp}. To show  $(ii) \implies (i)$,  by the Fekete--Schoenberg Theorem \ref{fec}, it suffices to show  all contiguous minors of $A$ of size $r \in [k]$ are positive. This is shown by induction on $r$. Let $r=1$ and let $A_1=(a_{ij})$ for some $i\in[m],~j \in [n]$. Then $x^{A_1}=(1)$ and $q^{A_1}=-(a_{ij})$. If $a_{ij}=0$, then $\lcpn{A_1, q^{A_1}}$ has infinitely many solutions, a contradiction. If $a_{ij}<0$, then $\lcpn{A_1,q^{A_1}}$ has two solutions, again a contradiction. Thus all $1\times 1$ minors of $A$ are positive. Let $r\in [k]\setminus\{1\}$ and suppose all contiguous minors of $A$ of size at most $(r-1)$ are positive. Let $A_r$ be an $r \times r$ contiguous submatrix of $A$, and define
	\[x^{A_r}:=(A^{11}_r,0,A^{13}_r,0,\ldots)^T, \qquad z^{A_r}:=(0,A^{12}_r,0,A^{14}_r,\ldots)^T,\qquad q^{A_r}:=-A_rx^{A_r}.\]   By the Fekete--Schoenberg Theorem \ref{fec}, all proper minors of $A_r$ are positive and so $x^{A_r},~z^{A_r}\geq 0$. Thus $x^{A_r}$ is a solution of $\lcpn{A_r, q^{A_r}}$.
	
	Next, observe that, if $A_r$ is singular, then \[A_r\begin{pmatrix}
	A^{11}_r\\0\\ A^{13}_r \\0 \\\vdots
	\end{pmatrix}=A_r \begin{pmatrix}
	0\\A^{12}_r\\ 0\\ A^{14}_r\\\vdots
	\end{pmatrix},
	\]
	so $z^{A_r}$ is another solution of $\lcpn{A_r,q^{A_r}}$, a contradiction. Thus $A_r$ is invertible.
	We now claim that $\det{A_r}>0$. Indeed, suppose $\det{A_r}<0$. Then
	\begin{center}
		$y=A_rz^{A_r}+q^{A_r}=-(\det{A_r})e^1 \geq 0$ \hbox{~~~ and~~~} $y^Tz^{A_r}=0$.
	\end{center}
	Thus, $z^{A_r}$ is again another solution of $\lcpn{A_r, q^{A_r}}$, a contradiction  by (ii). Hence $\det{A_r}>0$ and the proof is complete.
\end{proof}
\begin{rem}
	In Theorem \ref{tp-lcp_1}, instead of the vector $x^{A_r}= (A^{11}_r,0,A^{13}_r,0,\ldots)^T$, we can take ${x^i}^{A_r} := (A^{i1}_r,0,A^{i3}_r,0,\ldots)^T$ for odd $i\in [r]$, or a positive linear combination of some of these vectors. The proof is similar to that of Theorem \ref{tp-lcp_1}, where we define $q^{A_r}$ similarly as in \eqref{tlcp}.  
\end{rem}
\subsection{Totally non-negative matrices and the LCP}
We now turn our attention to identifying $TN_k$ matrices via the LCP. For a totally non-negative matrix $A\in \mathbb{R}^{n \times n}$, $\lcp{A}$ need not have a solution for some $q\in \rn{n}$. For instance, $A=\begin{pmatrix}
0 & 1\\0 & 0
\end{pmatrix}$ is a totally non-negative matrix, but $\lcp{A}$ has no solution for $q=\begin{pmatrix}
0 \\-1
\end{pmatrix}$.\\ For an $n \times n$ real matrix $A$, let $Q_A$ denote the set of all $q \in \mathbb{R}^n$ for which $\sol{A}\neq \emptyset$. If $A \in \mathbb{R}^{n\times n}$ is a matrix with non-negative entries, then $Q_A=\rn{n}$  if and only if all the diagonal entries of $A$ are positive \cite[Chapter~3.8]{CPS09}. Our next result gives a sufficient condition for total non-negativity via the Linear Complementarity Problem. To proceed further, we need a basic result characterizing $TN_k$, which was surprisingly discovered only recently.
\begin{theorem}\cite{CKK21}\label{tnsnr}
	Let $m,n \geq k \geq 1$ be integers. Given $A \in \mathbb{R}^{m \times
		n}$, the following statements are equivalent.
	\begin{enumerate}
		\item The matrix $A$ is totally non-negative of order $k$.
		
		\item Every square submatrix of $A$ of size $r \in [k]$ has the non-strict
		sign non-reversal property with respect to $\mathbb{R}^r$.
		\item Every square submatrix of $A$ of size $r \in [k]$ has the non-strict
		sign non-reversal property with respect to $\altr{r}$.
	\end{enumerate}
\end{theorem}
We now show how to apply the LCP to deduce total non-negativity:
\begin{prop}\label{tnlcp}
	Let $m,n \geq k \geq 1$ be integers and let $A \in \mathbb{R}^{m \times n}$. Then $A$ is totally non-negative of order $k$ if the following two conditions hold, for $1\leq r \leq k$: \begin{enumerate}
		\item($2\leq r \leq k$):  For every $r \times r$ submatrix $A_r$ of $A$ and for all $q \in Q_{A_r}$ with $q\leq 0$, if $z^1=(x_1,0,x_3,\ldots)^T$ and $z^2=(0,x_2,0,x_4,\ldots)^T$ (where all $x_i>0$) are two solutions of $\lcp{A_r}$, then $A_rz^1=A_rz^2$.
		\item($r=1$): For every $1\times 1$ submatrix $A_1$ of $A$ and for all scalars $q \in Q_{A_1}$, if $z^1$ and $z^2$ are two solutions of $\lcp{A_1}$, then $A_1z^1=A_1z^2$.
	\end{enumerate}
\end{prop}
\begin{proof}
	First we show that all $1\times 1$ minors of $A$ are non-negative. Let $A_{1}=(a_{ij})$ for some $i\in[m],~j\in [n]$. If $a_{ij}=0$, then we are done. If $a_{ij}<0$, then $\lcp{A_1}$ has two solutions $z^1=(1)$ and $z^2=(0)$, where $q=-(a_{ij})$, but $A_1z^1\neq A_1z^2$, a contradiction. Thus the matrix $A$ has non-negative entries.
	
	Next we claim that the determinant of every square submatrix $A_r$ of $A$ of size $r \in [k]\setminus\{1\}$ is non-negative. Fix $r \in [k]\setminus\{1\}$ and let $A_r$ be an $r \times r$ submatrix of $A$.  By Theorem \ref{tnsnr}, it is sufficient to show that $A_r$ has the non-strict sign non-reversal property with respect to $\altr{r}$. Suppose that $A_r$ does not satisfy this property. Then there exists $x \in \altr{r}$ such that $x_i(A_rx)_i<0$ for all $i \in [r]$. Defining $x^+,x^-,v^+,v^-$ and $q$ as in the proof of Theorem \ref{tp-lcp}, we conclude that $x^+$ and $x^-$ are two solutions of $\lcp{A_r}$. Also, $x^+,~x^-$ have sign patterns $\begin{pmatrix}
	+ \\ 0 \\+\\0\\\vdots
	\end{pmatrix}$ and $\begin{pmatrix}
	0 \\ + \\0\\+\\\vdots
	\end{pmatrix}$ respectively (or vice-versa), so $q\leq 0$ (since $A\geq 0$). Since $Ax \in \altr{r}$, by \eqref{thrmaeq1}, we have
	\begin{center}
		$A_rx^+ + q=v^+\neq v^-= A_rx^-+q.$
	\end{center}
	Thus $A_rx^+\neq A_rx^-$, a contradiction. Hence $A$ is $TN_k$. \end{proof}
The converse of the above result need not be true. We illustrate this with an example.
\begin{example}
	Let $A=\begin{pmatrix}
	2 & 1 & 1\\2 & 1 &1\\ 1 & 1 & 1
	\end{pmatrix}$ and $q=\begin{pmatrix}
	-3\\-3\\-2
	\end{pmatrix}$. Then $A$ is a totally non-negative matrix and $z^1=\begin{pmatrix}
	1\\0\\1
	\end{pmatrix}$ and $z^2=\begin{pmatrix}
	0\\3\\0
	\end{pmatrix}$ are two solutions of $\lcp{A}$, but $Az^1\neq Az^2$.
\end{example}
The next result improves on the previous one, by giving a sufficient condition for total non-negativity using the LCP at a single vector $q$ (for each square submatrix of $A$).
\begin{prop}\label{tnlcp1}
	Let $m,n \geq k \geq 1$ be integers and $A \in \mathbb{R}^{m \times n}$. Then $A$ is $TN_k$ if for every square submatrix $A_r$ of $A$ of size $r \in [k]$, if $z^1$ and $z^2$ are two solutions of $\lcpn{A_r,q^{A_r}}$ then $A_rz^1=A_rz^2$, where $q^{A_r}$ is defined as in \eqref{tlcp}.
\end{prop}
\begin{proof}
	We show that $\det {A_r}\geq 0$ for all $r \times r$ submatrices $A_r$ of $A$, by induction on $r \in [k]$. The base case $r=1$ can be proved similarly to Proposition \ref{tnlcp}.
	
	Let $r \in [k]\setminus \{1\}$ and suppose that all the minors of $A$ of size at most $(r-1)$ are non-negative. Let $A_r$ be a square submatrix of $A$ of size $r$. If $\det {A_r}=0$, then we are done. If $\det {A_r}<0$, repeating the proof of Theorem \ref{tp-lcp_1}, once again we have $x^{A_r}=\begin{pmatrix}
	A^{11}_r\\0\\ A^{13}_r \\0 \\\vdots
	\end{pmatrix}$ and $z^{A_r}= \begin{pmatrix}
	0\\A^{12}_r\\ 0\\ A^{14}_r\\\vdots
	\end{pmatrix}$ are two distinct solutions of $\lcpn{A_r,q^{A_r}}$, but $Ax^{A_r}\neq Az^{A_r}$. Thus $\det {A_r}>0$ and hence $A$ is $TN_k$.
\end{proof}
\begin{rem}
	The converse of the preceding proposition need not be true. For instance, $A=\begin{pmatrix}
	0 & 0 & 0\\2 & 2 &1\\ 1 & 1 & 1
	\end{pmatrix}$ is a totally non-negative matrix and $q^{A}=-\begin{pmatrix}
	0\\2\\1
	\end{pmatrix}$. Then $z^1=\begin{pmatrix}
	1\\0\\0
	\end{pmatrix}$ and $z^2=\begin{pmatrix}
	2\\0\\0
	\end{pmatrix}$ are two solutions of $\lcpn{A,q^A}$, but $Az^1\neq Az^2$.
\end{rem}
In the last part of this section, we discuss the solution set of $\lcp{A}$, with $A \in \mathbb{R}^{n \times n}$ being a totally non-negative matrix. First we recall a 1968 result of Karlin for nonsingular totally non-negative matrices.
\begin{lemma}[(Karlin,~\cite{K68})]\label{nontn}
	Let $A \in \mathbb{R}^{n \times n}$ be a nonsingular totally non-negative matrix. Then all the principal minors of $A$ are positive.
\end{lemma}

Thus, if $A$ is a nonsingular totally non-negative matrix, by Theorem \ref{lcpp}, $\lcp{A}$ has a unique solution for all $q\in \rn{n}$.

Given a matrix $A \in \mathbb{R}^{n \times n}$, a solution $x$ of $\lcp{A}$ is called a \textit{nondegenerate solution} if $x_i \neq (Ax+q)_i$ for all $i \in [n]$. Note that a solution $x$ of the $\lcp{A}$ is nondegenerate if and only if the support of $x$ and the support of $y:=Ax+q$ are complementary index sets in $[n]$. A vector $q \in \mathbb{R}^n$ is called \textit{nondegenerate} with respect to $A$ if all the solutions of $\lcp{A}$ are nondegenerate. With this information in hand, we now present a partial converse of Proposition \ref{tnlcp}.

\begin{lemma}
	Let $A \in \mathbb{R}^{n \times n}$ be a totally non-negative matrix and let $q \in \mathbb{R}^n$ be a nondegenerate vector with respect to $A$. If $z^1=(x_1,0,x_3,\ldots)^T$ and $z^2=(0,x_2,0,x_4,\ldots)^T$ are two solutions of $\lcp{A}$ with all $x_i>0$, then $A_rz^1=A_rz^2$.
\end{lemma}
\begin{proof}
	Suppose that $A \in \mathbb{R}^{n \times n}$ is a totally non-negative matrix and $q$ is a nondegenerate vector with respect to $A$. Let $z^1=(x_1,0,x_3,\ldots)^T$ and $z^2=(0,x_2,0,x_4,\ldots)^T$ be two solutions of $\lcp{A}$ such that all $x_i>0$ and $Az^1\neq Az^2$. Let $y^i=Az^i+q$ for $i=1,2$ and let $z=z^1-z^2$ and $y=y^1-y^2$. Then $y=Az$ and $y,z \in \altr{n}$, since $q$ is nondegenerate. Since $z^1$ and $z^2$ are solutions of $\lcp{A}$, $\sgn(y_j)=-\sgn(z_j)$ for all $j \in [n]$. Thus $z_i(Az)_i<0$ for all $i \in [n]$, a contradiction by Theorem \ref{tnsnr}. Hence $Az^1=Az^2$.
\end{proof}
In the next result we put certain conditions on the matrix $A$ instead of the vector $q$, and present another partial converse of Proposition \ref{tnlcp}.
\begin{theorem}
	Let $A \in \mathbb{R}^{n \times n}$ be a totally non-negative matrix such that whenever a set of columns of $A$ forms a basis of its column space, the corresponding principal submatrix is invertible. For all $q \in Q_A$, if $z^1$ and $z^2$ are two solutions of $\lcp{A}$, then $Az^1=Az^2$.
\end{theorem}
\begin{proof}
	We prove this by contradiction. Let $q \in Q_A$ and let $z^1$ and
	$z^2$ be two distinct solutions of $\lcp{A}$ such that $Az^1 \neq
	Az^2$. Let $z=z^1-z^2$ and $y=y^1-y^2$, where $y^1:=Az^1+q$ and
	$y^2:=Az^2+q$. By the definition of the LCP, $z^k_i
	y^k_i = 0$ and $z^k_i, y^k_i \geq 0$ for $k=1,2$ and all $i$, so
	\[z_iy_i=(z^1_i-z^2_i)(y^1_i-y^2_i)=-z^1_iy^2_i-z^2_iy^1_i\leq 0
	\hbox{~~for~} i\in [n].\] Also, $y\neq 0$, since $Az^1\neq
	Az^2$. Let $J_z \subseteq [n]$ denote the support of $z$. Then
	\begin{equation}
	y_i\leq 0 ~~(\geq 0) \hbox{~~whenever~~} z_i>0 ~~(<0) \hbox{~~for all~~}i \in J_z \label{tnconveq1}.
	\end{equation}
	
	We first claim that there exists $x \in \mathbb{R}^n$ with support $J_x\subseteq J_z$ such that $\sgn(z_i)=\sgn(x_i)$ for $i \in J_x$, $Ax=y$, and the columns of $A$ corresponding to $J_x$ are linearly independent. If the columns of $A$ corresponding to $J_z$ are linearly independent then we are done. Otherwise there exists a nonzero vector $v \in \mathbb{R}^n$ such that $Av=0$ and $v_j=0$ for $j\in [n] \setminus J_z$. Define
	\begin{equation}
	\alpha :=\min \left\{ \frac{\vert z_i\vert}{\vert v_i\vert}: i \in J_z ~\hbox{and}~ v_i \neq 0 \right\}, \qquad x^1:=z-\sgn(z_{i_0})\sgn(v_{i_0})\alpha v,
	\end{equation} where $i_{0}\in J_z$ is the index where the minimum $\alpha$ is attained. Then $Ax^1=y$ and $\sgn(z_i)=\sgn(x^1_i)$ for $i \in J_{x^1}$. Since $x^1_{i_{0}}=\sgn(z_{i_0})\left[\vert z_{i_{0}} \vert -\sgn(v_{i_0})\alpha v_{i_{0}}\right]$, for at least one component $i \in J_z$, $x^1_i=0$. Now, if the columns of $A$ corresponding to $J_{x^1}$ are linearly independent, then we are done. Otherwise, we apply the same technique to the new vector and keep repeating. Thus there exists $x \in \mathbb{R}^n$ with support $J_x \subsetneq J_z$ such that\begin{equation} \sgn(z_i)=\sgn(x_i)  \hbox{~~for~~} i \in J_x, \qquad Ax=y\label{tnconveq2}
	\end{equation}
	and the columns of $A$ corresponding to $J_x$ are linearly independent. Since $A$ is totally non-negative, by the hypothesis and Lemma \ref{nontn}, all the principal minors of $A_{J_x}$ are positive, where $A_{J_x}$ denotes the principal submatrix of $A$ whose rows and columns are indexed by $J_x$. By \eqref{tnconveq1} and \eqref{tnconveq2}, $A_{J_x}$ reverses the sign of $x_{J_x}$, a contradiction by \cite[Theorem 2]{gale-nikai-pmat}. Thus $Az^1=Az^2$.
\end{proof}

The next remark, which is a standalone observation that may
be of independent interest, suggests steps to solve the problem $\lcp{A}$
with $A$ being a $TP/TN$ matrix.
\begin{rem} 
	Let $A \in \mathbb{R}^{n \times n}$ be a totally positive/non-negative matrix and let $q \in \rn{n}$. If a coordinate of $q$ is non-negative, say $q_1\geq 0$, then we obtain the submatrix $B$ from $A$ by deleting the corresponding row and column of $A$, and similarly we obtain a vector $q^2$ from $q$. Next we try to solve the new Linear Complementarity Problem $\lcpn{B,q^2}$. If $x^2$ is a solution of the new problem $\lcpn{B,q^2}$ then $x:=(0,{(x^2)}^T)^T$ is a solution of the original problem $\lcp{A}$.
\end{rem}

\section{Theorem \ref{TPsnr}: Sign non-reversal property for totally positive matrices}\label{snrsec}

In recent joint work \cite{CKK21}, Theorem \ref{tnsnr} above had a fourth part in terms of a \textit{single} vector, for characterizing $TN_k$ matrices. Similarly, we gave a new test for total positivity using the sign non-reversal property at a single vector, but under certain additional conditions:
\begin{theorem}\cite{CKK21}\label{tp-sign-rev_k_last}
	Let $m,n \geq k \geq 1$ be integers.
	Given $A \in \mathbb{R}^{m \times n}$, the following statements
	are equivalent.
	\begin{enumerate}
		\item[(1)] The matrix $A$ is totally positive of order $k$.
		\item[(2)] For every $r \in [k]$ and contiguous $r
		\times r$ submatrix $A_r$ of $A$, define the vectors
		\begin{equation}\label{Ezbold}
		\bd{r}:=(1,-1,1,\ldots (-1)^{r-1})^T,\qquad z^{A_r} := (\det{A_r}) \adj(A_r) \bd{r}.
		\end{equation}
		Now:
		(i)~$A_rx \neq 0$ for all $x \in \altr{r}$; and
		(ii)~$A_r$ has the non-strict sign non-reversal property with respect to $z^{A_r}$.
	\end{enumerate}
\end{theorem}
In the next result, we drop the condition $(i)$ i.e, $A_rx \neq 0$ for all $x \in \altr{r}$, and identify a new test vector for the  sign non-reversal property which is simpler than \eqref{Ezbold}. In particular, we are able to characterize total positive matrices using the sign non-reversal property truly at a single vector.
\begin{utheorem}\label{TPsnr}
	Let $m,n \geq k \geq 1$ be integers.
	Given $A \in \mathbb{R}^{m \times n}$, the following statements
	are equivalent.
	\begin{enumerate}
		\item The matrix $A$ is totally positive of order $k$.
		\item For every $r \in [n]$ and contiguous $r
		\times r$ submatrix $A_r$ of $A$, define the vector
		\begin{equation}\label{Ezb}
		x^{A_r}:= (A^{11}_r,-A^{12}_r,\ldots , (-1)^{r-1}A^{1 r}_r)^T.
		\end{equation}
		Then $A_r$ has the sign non-reversal property with respect to $x^{A_r}$.
	\end{enumerate}
\end{utheorem}
\begin{proof}
	That $(i) \implies (ii)$ is immediate from Theorem \ref{tp-sign-rev_k} (iii).
	
	We will now prove $(ii) \implies (i)$ using induction on the size of the contiguous minors of $A$ (by the Fekete--Schoenberg Theorem \ref{fec}). The base case $r=1$ directly follows from the hypothesis in $(ii)$. For the induction step, assume that all contiguous minors of $A$ of size at most $r-1$ (where $2\leq r\leq k)$ are positive. Let $A_r$ be an $r\times r$ contiguous submatrix of $A$. By the induction hypothesis, all the proper contiguous minors of $A_r$ are positive. Thus all the proper minors of $A_r$ are positive by the Fekete--Schoenberg Theorem
	\ref{fec}.

	Define the vector $x^{A_r}$ as in \eqref{Ezb}. Then $x^{A_r}\in \altr{r}$, and 
	\[
	A_rx^{A_r}=(\det{A_r})e^1.
	\]
	
	Now we claim that $\det{A_r}>0$. By hypothesis, there exists $i \in [r]$ such that
	\begin{equation}
	0< x^{A_r}_i(A_r x^{A_r})_i=(\det{A_r}) x^{A_r}_i e^1_i. \nonumber
	\end{equation}
	Thus $i=1$ and $\det{A_r}>0$.
\end{proof}
\begin{rem}\label{snrtnrem}
	Note that the vector $x^{A_r}$ as defined in \eqref{Ezb} is the first column of $\adj(A_r)$. Instead of $\adj(A_r)e^1$ one can take $x^{A_r}=\adj(A_r)e^j$ for any $j \in[r]$, or even $x^{A_r}:=\adj(A_r)\alpha$, where $\alpha = (\alpha_1, -\alpha_2, \alpha_3, \dots, (-1)^{r-1} \alpha_r)^T$ is an arbitrary nonzero vector in the orthant where all $\alpha_i \geq 0$ (or $\leq 0$). Then Theorem \ref{TPsnr} still holds, with a similar proof.
\end{rem}

\begin{rem}
	Recently in \cite{CKK21}, we discussed a new characterization of totally non-negative matrices in terms of the non-strict sign non-reversal property. We proved that $ A \in \mathbb{R}^{m \times n}$ is totally non-negative  of order $k$ if and only if every submatrix $A_r$ of $A$ of size $r\in [k]$ has the non-strict sign reversal property with respect to a single vector of the form $z^{A_r}=(\det{A_r})\adj(A_r)\bd{r}$ as in \eqref{Ezbold}. It is easy to verify that the result is still true if we take any non-negative integer power of $\det{A_r}$. More generally, the result holds if we take $z^{A_r}:=\adj(A_r)\alpha$ for arbitrary fixed $\alpha \in \altr{r}$.
\end{rem}

\section{Theorem \ref{TP-vandimnew}: Variation diminution and total positivity}\label{vdsec}
A very important and widely used characterization of totally positive and totally non-negative matrices is in terms of their variation diminishing property. The term “variation diminishing” was coined by P\'olya in correspondence with Fekete in 1912 \cite{FP12} to prove the following result (stated by Laguerre \cite{Laguerre}) using  P\'olya frequency sequences and their
variation diminishing property: \textit{given a polynomial $f(x)$ 
	and an integer $s \geq 0$, the number $var(e^{sx} f(x))$ of variations in the
	Maclaurin coefficients of $e^{sx} f(x)$ is non-increasing in $s$, hence
	is bounded above by $var(f) < \infty$.} The variation diminishing property of totally non-negative matrices was first studied by Schoenberg \cite{S30} in 1930. In 1950~\cite{GK50}, Gantmacher--Krein
made fundamental contributions relating total positivity and variation diminution. To proceed, we need some notation.
\begin{defn}
	Given a vector $x \in \mathbb{R}^n$, let $S^-(x)$ denote the number of changes in sign after deleting all zero entries in $x$. Next, the zero entries of $x$ are arbitrarily assigned a value of $\pm 1$, and we denote by $S^+(x)$ the maximum possible number of sign changes in the resulting sequence. For $0\in \mathbb{R}^n$, we set $S^+(0):=n$ and $S^-(0):=0$.
\end{defn}
The following result of Brown--Johnstone--MacGibbon \cite{BJM81} (see also Gantmacher--Krein \cite[Chapter V]{GK50}) gives a characterization of totally positive matrices in terms of the variation diminishing property (cited from Pinkus's book).
\begin{theorem}[{\cite[Theorem 3.3]{pinkus}}]\label{TP-vandim}
	Given a real $m \times n$ matrix $A$, the following statements are
	equivalent.
	\begin{enumerate}
		\item $A$ is totally positive.
		\item For all $0\neq x \in \mathbb{R}^n$, $S^+(Ax) \leq S^-(x)$. If
		moreover equality occurs and $Ax \neq 0$, the first (last) component of $Ax$ (if zero, the unique sign required to determine $S^+(Ax)$) has the same sign as the first (last)
		nonzero component of $x$. 
	\end{enumerate}
\end{theorem}
We next recall the analogous characterization for totally non-negative matrices using the variation diminishing property.
\begin{theorem}[{\cite[Theorem 3.4]{pinkus}}]\label{TN-vandim}
	Given a real $m \times n$ matrix $A$, the following statements are
	equivalent.
	\begin{enumerate}
		\item $A$ is totally non-negative.
		\item For all $x \in \mathbb{R}^n$, $S^-(Ax) \leq S^-(x)$. If
		moreover equality occurs and $Ax \neq 0$, the first (last)
		nonzero component of $Ax$ has the same sign as the first (last)
		nonzero component of $x$.
		
	\end{enumerate}
\end{theorem}
Observe that the set of test vectors in the second statements of Theorems \ref{TP-vandim} and \ref{TN-vandim} is uncountable. It is natural to ask if this can be reduced to a finite set of test vectors? Our next result provides a positive answer -- in fact, a \textit{single} vector for each submatrix.

\begin{utheorem}\label{TP-vandimnew}
	Given a real $m \times n$ matrix $A$, the following statements are
	equivalent.
	\begin{enumerate}
		\item $A$ is totally positive.
		\item For all $0 \neq x \in \mathbb{R}^n$, $S^+(Ax) \leq S^-(x)$. If
		moreover equality occurs and $Ax \neq 0$, the first (last) component of $Ax$ (if zero, then the unique sign required to determine $S^+(Ax)$) has the same sign as the first (last) nonzero component of $x$. 
		\item For every $r \in[\min\{m,n\}]$ and contiguous $r
		\times r$ submatrix $A_r$ of $A$, define the vector
		\begin{equation}\label{vdeq1}
		x^{A_r}:= (A^{11}_r,-A^{12}_r,\ldots ,(-1)^{r-1}A^{1 r}_r)^T.
		\end{equation}
		Then $S^+(A_rx^{A_r}) \leq S^-(x^{A_r})$. If
		equality holds here, then the first (last)
		component of $A_rx^{A_r}$ (if zero, the unique sign required to determine $S^+(A_rx^{A_r})$) has the same sign as the first (last)
		nonzero component of $x^{A_r}$. 
	\end{enumerate}
\end{utheorem}
\begin{proof}
	That $(i)\implies (ii)$ was Theorem \ref{TP-vandim}. We first show $(ii) \implies (iii)$. Let $k=\min\{m,n\}$, fix $r \in [k]$ and let $A_r$ be a contiguous submatrix of $A$, say $A_r=A_{I \times J}$ for contiguous sets of indices $I \subseteq [m]$ and $J \subseteq [n]$ with $\lvert I\rvert = \lvert J \rvert =r$. Define $x^{A_r}$ as in \eqref{vdeq1}; note that this is nonzero. We extend $x^{A_r}$ to $x\in \mathbb{R}^n$ by embedding in positions $J$ and padding by zeroes elsewhere. Then 
	\begin{equation} S^-(x)=S^-(x^{A_r}) \hbox{~~and~~} S^+(Ax)\geq S^+(A_rx^{A_r}).\label{vdeq2}
	\end{equation} Thus  $S^+(A_rx^{A_r}) \leq S^-(x^{A_r})$, since $S^+(Ax)\leq S^-(x)$.
	
	Suppose that $S^+(A_rx^{A_r}) = S^-(x^{A_r})$. Then $S^+(Ax)=S^-(x)$. Without loss of generality, assume that  the first and last nonzero entries of $x^{A_r}$ are in positions $s,t\in [r]$, respectively. Enumerate the indices in $i \in [r]$ by $i_1<i_2<\cdots <i_r$. By the hypothesis, it follows that all coordinates of $Ax$ in positions $1,2,\ldots, i_s ~(\hbox{respectively}~i_t,\ldots, m)$ have the same sign and this sign agrees with that of $x^{A_r}_s ~(\hbox{respectively}~x^{A_r}_t)$. This concludes  $(ii) \implies (iii)$.
	
	To show $(iii) \implies (i)$, by Theorem \ref{fec}, it suffices to show that the determinants of  all $r \times r$ contiguous submatrices are positive, for $1\leq r \leq \min\{m,n\}$. We prove this by induction on $r$. The case $r=1$ is immediate from $(iii)$. For the induction step, suppose that all contiguous minors of $A$ of size at most $(r-1)$ are positive and $A_r$ is an $r \times r$ contiguous submatrix of $A$. By the Fekete--Schoenberg Theorem \ref{fec}, $A_r$ is $TP_{r-1}$. Define the vector $x^{A_r}$ as in \eqref{vdeq1}. Then $x^{A_r}\in \altr{r}$ and $S^-(x^{A_r})=r-1$.
	
	We first show that $A_r$ is invertible. Indeed suppose that $A_r$ is singular. Then $A_r x^{A_r}=0$ and $r=S^+(A_rx^{A_r})> S^-(x^{A_r})$, a contradiction. Thus $A_r$ is invertible. 
	
	Next we show that $\det {A_r}>0$. Since $A_r x^{A_r}=(\det{A_r})e^1$, we have
	\[r-1=S^+(A_rx^{A_r})= S^-(x^{A_r}).
	\]
	Thus by $(iii)$, the first component of $A_rx^{A_r}$ has the same sign as the first nonzero component of $x^{A_r}$. Hence $\det{A_r}>0$ and the induction step is complete.
\end{proof}
\begin{rem}
	Remark \ref{snrtnrem} applies verbatim to Theorem \ref{TP-vandimnew}.
\end{rem}
We conclude this section with a similar improvement to the classical characterization of $TN$ matrices via variation diminution. \begin{theorem}\label{TN-vandimnew}
	Given a real $m \times n$ matrix $A$, the following statements are
	equivalent.
	\begin{enumerate}
		\item $A$ is totally non-negative.
		\item For all $x \in \mathbb{R}^n$, $S^-(Ax) \leq S^-(x)$. If
		moreover equality occurs and $Ax \neq 0$, then the first (last)
		nonzero component of $Ax$ has the same sign as the first (last)
		nonzero component of $x$. 
		\item For every square submatrix $A_r$ of $A$ of size $r \in [\min\{m,n\}]$ define the vector
		\begin{equation}\label{vdeqtn1}
		y^{A_r}:= \adj(A_r)\alpha, \hbox{~~for arbitrary fixed}~ \alpha \in \altr{r}.
		\end{equation}
		Then $S^-(A_ry^{A_r}) \leq S^-(y^{A_r})$. If equality holds here and $A_ry^{A_r} \neq 0$, the first (last)
		nonzero component of $A_ry^{A_r}$ has the same sign as the first (last)
		nonzero component of $y^{A_r}$.
	\end{enumerate}
	\begin{proof}
		That $(i)\implies (ii)$ was Theorem \ref{TN-vandim}. To show $(ii)\implies (iii)$, repeat the proof of Theorem \ref{TP-vandimnew}, but working with arbitrary $r \times r$ submatrices $A_r$ of $A$, where $r \in [\min\{m,n\}]$ and the vector $y^{A_r}$ from \eqref{vdeqtn1} is used in place of $x^{A_r}$.
		
		Next we show $(iii) \implies (i)$. Let $k=\min\{m,n\}$ and $r \in [k]$. We show that $\det{A_r}\geq 0$ for all square submatrices $A_r$ of $A$ of size $r$. We prove this by induction on $r$, with the base case $r=1$ immediate. Now suppose all minors of $A$ of size at most $(r-1)$ are non-negative and $A_r$ is an $r \times r$ submatrix of $A$. If $\det{A_r}=0$ then we are done. Let $\det{A_r} \neq 0$, $\alpha=(\alpha_1,\alpha_2\ldots, \alpha_r)^T \in \altr{r}$ and define $y^{A_r}$ as in \eqref{vdeqtn1}. Then no row of $\adj(A_r)$ is zero and $y^{A_r}\in \altr{r}$. Since $A_ry^{A_r}=(\det{A_r}) \alpha$, we have 
		\[
		S^-(A_ry^{A_r})=S^-(y^{A_r})=r-1.
		\]
		Also, the first entry of $A_ry^{A_r}$ is $\alpha_1\det{A_r}$ and the first entry of $y^{A_r}$ is
		\[
		\sum_{j=1}^r (-1)^{j-1}\alpha_j  {A_r}^{j 1},
		\]
		where the summation is positive  (or negative)  if and only if $\alpha_1$ is positive (or negative) by the induction hypothesis. Thus $\det{A_r}>0$ by assumption.
	\end{proof}
\end{theorem}
\section{Test vectors from any other orthant do not work}

In the previous two sections, we have seen that $TP_k$ matrices are characterized by their contiguous square submatrices $A_{r}$ ($r \in [k]$) satisfying either the variation diminishing property, or the sign non-reversal property, on the entire open bi-orthant $\altr{r}$ for each $r \in [k]$ -- or on single test vectors which turn out to lie in this bi-orthant. We conclude by explaining the sense in which these results are `best possible'. Informally, we claim that if $x \in \mathbb{R}^r$ lies in any other open orthant (i.e., all $x_j \neq 0$ and there are two successive $x_j$ of the same sign), then \textit{every} $TP_{r-1}$ matrix $A_{r}$ of size $r$ satisfies the variation diminishing property and the sign non-reversal property with respect to $x$. (In particular, since there exist $TP_{r-1}$ matrices $A_{r}$ that are not $TP$, the aforementioned characterizations cannot hold with test vectors in any open bi-orthant other than in $\altr{r}$.)

\begin{theorem}\label{thrmlast}
	Suppose $x \in \mathbb{R}^r$ has nonzero coordinates, and at least two successive coordinates have common sign. Let $A_r \in \mathbb{R}^{r \times r}$ be $TP_{r-1}$. Then:
	\begin{enumerate}
		\item[(1)] $A_r$ satisfies the variation diminishing property with respect to $x$. In other words, $S^+(A_rx) \leq S^-(x)$. If
		moreover equality holds, then the first (last) component of $Ax$ (if zero, the unique sign required to determine $S^+(Ax)$) has the same sign as the first (last)
		component of $x$. 
		\item[(2)] $A_r$ satisfies the sign non-reversal property with respect to $x$. In other words, there exists a coordinate $i \in [r]$ such that $x_i(A_rx)_i>0$.
	\end{enumerate}
\end{theorem}

It remains to observe that there do exist matrices which are $TP_{n-1}$ but not $TP_n$, for each $n$. This follows from the analysis of a P\'olya frequency function studied by Karlin \cite{Karlin64} in 1964; this analysis was only recently carried out by Khare \cite{Khare20}. See Example \ref{examplelast} below, for details.

\begin{proof}[of Theorem \ref{thrmlast}]
	Let $x \in \mathbb{R}^r$ such that all the coordinates of $x$ are nonzero, and at least two successive coordinates have common sign. Decompose $x$ into contiguous coordinates of like signs:	
	\begin{equation}\label{partnlast}
	(x_1,\ldots ,x_{s_1}),~~(x_{s_1 +1},\ldots ,x_{s_2}), ~~ \ldots (x_{s_k +1},\ldots ,x_{r}),
	\end{equation}
	with all coordinates in the $i$th component having the same sign, which we choose to be $(-1)^{i-1}$ without loss of generality. Set $s_0=0$ and $s_{k+1}=r$ and observe that $k \leq r-2$.  Let $a^1,\ldots, a^r \in \mathbb{R}^r$ denote the columns of $A_r$, and define
	\[
	b^i:= \sum\limits_{j=s_{i-1}+1}^{s_i}\vert x_j \vert a^j,\quad \hbox{~for~} i\in [{k+1}].
	\]
	We claim that the matrix $B:=[b^1,\ldots, b^{k+1}] \in \mathbb{R}^{r \times (k+1)}$ is totally positive. Indeed, since all $x_j$ are nonzero, and all
	proper minors of $A_r$ are positive, given an integer $p \in [k+1]$ and
	$p$-element subsets $I \subset [r]$, $J\subseteq
	[k+1]$ using standard properties of
	determinants, we have
	\[
	\det B_{I\times J}= \sum_{l_1=s_{j_1-1}+1}^{s_{j_1}} \cdots
	\sum_{l_p=s_{j_p-1}+1}^{s_{j_p}}  \vert x_{l_1} \vert \ldots \vert
	x_{l_p} \vert \det {A_r}_{I\times L} > 0,
	\]
	where $L=\{l_1,\ldots l_p\}$, and $B_{I \times J}$ denotes the
	submatrix of $B$ whose rows and columns are indexed by $I,J$ respectively. Thus $B$ is $TP$; and we also define $y:=Ax=B\bd{k+1}$, where $\bd{k+1}:=(1,-1,1,\ldots (-1)^{k})^T\in \altr{k+1}$.
	
	With this analysis in hand, we can now prove the theorem.
	
	\noindent$\textbf{(1).}$ Note that $S^-(x)=k$. If $S^+(Ax)>S^-(x)$, then there exist indices $i_1<i_2<\cdots < i_{k+2}\in [r]$ and a sign $\epsilon=\pm 1$ such that $(-1)^{t-1}\epsilon y_{i_t}\geq 0$ for $t \in [{k+2}]$. Moreover at least two of the $y_{i_t}$ are nonzero, since $B$ is $TP$. Define the $(k+2)\times (k+2)$ matrix \begin{center}$M:=[y_I|B_{I \times [k+1]}], \hbox{ where } I=\{i_1,\ldots, i_{k+2}\}.$\end{center}
	Then $\det M=0$, since the first column of $M$ is an alternating sum of the rest. Expanding along the first column,
	\[
	0=\sum_{t=1}^{k+2} (-1)^{t-1}y_{i_t} \det B_{{I\setminus \{i_t\}}\times [k+1]},
	\]
	
	\noindent a contradiction, since all terms $(-1)^{t-1}y_{i_t}$ have the same sign, at least two $y_{i_t}$ are nonzero, and all minors of $B$ are positive. Thus $S^+(Ax)\leq S^-(x)$.
	
	It remains to prove the remainder of  the assertion $(1)$. We continue to employ the notation in the preceding discussion, now using $k+1$ in place of $k+2$. We claim that, if $S^+(Ax)=S^-(x)=k$ with $Ax\neq 0$, and if moreover $(-1)^{t-1}\epsilon y_{i_t}\geq 0$ for $t \in [{k+1}]$, then $\epsilon=1$.
	
	To show this we use the fact that the submatrix $B_{I\times [k+1]}$ is $TP$, where $I=\{i_1,\ldots,i_{k+1}\}$. Also, $B_{I\times [k+1]}\bd{k+1}=y_I$, since $B\bd{k+1}=Ax$. By Cramer's rule, the first coordinate of $\bd{k+1}$ is
	\[
	1=\frac{\det [y_I|B_{I\times [k+1]\setminus \{1\}}]}{\det B_{I\times [k+1]}}.
	\]
	Multiply both sides by $\epsilon \det B_{I\times [k+1]}$ and expand the numerator along the first column. This yields:
	
	\[
	\epsilon \det B_{I\times [k+1]}=\sum_{t=1}^{k+1} (-1)^{t-1}\epsilon y_{i_t} \det B_{I\setminus \{i_t\}\times [k+1]\setminus \{1\}}.
	\]
	Since each summand on the right side is non-negative with at least one positive  and $B$ is $TP$, it implies that $\epsilon=1$.
	
	\noindent$\textbf{(2).}$ We prove this by contradiction. Let $x_iy_i\leq 0$ for all $i \in [r]$. Consider the index set \begin{center} $I=\{i_1,\ldots, i_{k+1}\}, \hbox{ where } i_t \in [s_{t-1}+1,s_t]$.
	\end{center}
	Then the matrix $B_{I\times [k+1]}$ is totally positive and $B_{I\times [k+1]}\bd{k+1}=y_I$.  Thus $B_{I\times [k+1]}$ reverses the signs of $\bd{k+1}$, a contradiction by Theorem \ref{tp-sign-rev_k}. 
\end{proof}

\begin{example}\label{examplelast}
	We now explain how to construct a multi-parameter family of real $n \times n$ matrices for every integer $n \geq 3$, each of which is $TP_{n-1}$ but has negative determinant. This construction involves non-integer powers of a certain P\'olya frequency function, studied by Karlin in 1964 \cite{Karlin64}:
	\begin{equation}\label{EBmatrix}
	\omg(x) := \begin{cases}
	xe^{-x}, & \text{if } x>0;\\
	0, \qquad & \text{otherwise}.
	\end{cases}
	\end{equation}
	Karlin showed that if $\alpha \in \mathbb{Z}^{\geq 0}\cup [k-2, \infty)$ then $\omg(x)^\alpha$ is $TN_k$, i.e., given  real $x_1<\cdots <x_k$ and $y_1<\cdots<y_k$, the matrix $B:=\left(\omg (x_i-y_j)^\alpha\right)^k_{i,j=1}$ is $TN$.
	
	Recently in \cite{Khare20}, Khare showed that if $y_1<\cdots < y_k< x_1 <\cdots <x_k$, then $\left(\omg (x_i-y_j)^\alpha\right)^k_{i,j=1}$ is $TP_k$ if $\alpha >k-2$ and not $TN_k$ if $\alpha \in (0,k-2)\setminus \mathbb{Z}$. Thus, consider $n\geq 3$ and $\alpha \in ({n-3},{n-2})$, and choose  real scalars $y_1<\cdots < y_n< x_1 <\cdots <x_n$.  Then the matrix $A:=\left(\omg (x_i-y_j)^\alpha\right)^n_{i,j=1}$ is $TP_{n-1}$ but not $TN_n$, whence $\det {A}<0$.
\end{example}
An analogue of Theorem \ref{thrmlast} holds for $TN_{r-1}$ matrices.
\begin{theorem}\label{Ttnlast}
	Let $x \in \mathbb{R}^r\setminus \altr{r}$ with all $x_i \neq 0$ and $A_r\in \mathbb{R}^{r \times r}$ be a $TN_{r-1}$ matrix. Then:

	\begin{enumerate}
		\item[(1)] $A_r$ satisfies the variation diminishing property with respect to $x$. In other words, $S^-(A_rx) \leq S^-(x)$. If
		moreover equality holds and $Ax \neq 0$, then the first (last)
		nonzero component of $Ax$ has the same sign as the first (last) component of $x$.  
		\item[(2)] $A_r$ satisfies the non-strict sign non-reversal property with respect to $x$. In other words, there exists a coordinate $i \in [r]$ such that   $x_i(A_rx)_i\geq0$. 
	\end{enumerate}
\end{theorem}
The proof requires  Whitney's density result for totally positive matrices and a  lemma on sign changes of limits of vectors.
\begin{theorem}[(Whitney,~\cite{Whitney})]\label{Twhitney}
	Given integers $m,n \geq k \geq 1$, the set of $m \times n$ $TP_k$
	matrices is dense in the set of $m \times n$ $TN_k$ matrices.
\end{theorem}
\begin{lemma}\cite{pinkus}\label{limsc}
	Given $x=(x_1,\ldots,x_n)^T\in \mathbb{R}^n\setminus \{0\}$, define the vector
	\[
	\overline{x}:=(x_1,-x_2,x_3,\ldots, (-1)^{n-1}x_n)\in \mathbb{R}^n.
	\]
	Then $S^+(x)+S^-(\overline{x})=n-1.$
	Moreover, if $\lim\limits_{p \to \infty} x_p =x$, then
	
	\[\liminf\limits_{p \to \infty} S^-(x_p)\geq S^-(x), \qquad \limsup\limits_{p \to \infty} S^+(x_p)\leq S^+(x).\]
\end{lemma}
	 We can now prove the above properties of $TN_{r-1}$ matrices.
	 \begin{proof}[of Theorem \ref{Ttnlast}]
	 	\textbf{(1).} Let $x \in \mathbb{R}^r\setminus \altr{r}$ with all $x_i \neq 0$. Since $A_r$ is totally non-negative of order $r-1$, by Whitney's density Theorem \ref{Twhitney}, there exists a sequence $A_r^{(l)}$ of totally positive matrices of order $r-1$ such that 
	 	\[\lim\limits_{l \to \infty} A_r^{(l)} =A_r.\]
	 	Now use Theorem \ref{thrmlast} and Lemma \ref{limsc} to compute:
	 	\[S^-(A_rx)\leq \liminf\limits_{l \to \infty} S^-(A_r^{(l)} x)\leq \liminf\limits_{l \to \infty} S^+(A_r^{(l)}x)\leq \liminf\limits_{l \to \infty} S^-(x)=S^-(x).\]
	 	Next, if equality occurs and $A_rx\neq 0$, then for all $l$ large enough, we have
	 	\[S^-(A_rx)\leq S^-(A_r^{(l)}x)\leq S^+(A_r^{(l)}x)\leq S^-(x),\]
	 	by Theorem \ref{thrmlast} and Lemma \ref{limsc}. Thus
		$S^-(A_r^{(l)}x)=S^+(A_r^{(l)}x)=S^-(x)$ for $l$
		sufficiently large. This implies (for large $l$) the sign
		changes in $A_r^{(l)}x$  have no dependence on the zero
		entries.  Thus the nonzero sign patterns of $A_r^{(l)} x$
		agree with those of $A_rx$. Also, by Theorem
		\ref{thrmlast}, both $x$ and $A_r^{(l)}x$ admit
		partitions of the form \eqref{partnlast} with alternating
		signs, with precisely $S^-(x)$-many sign changes.  Hence
		the same holds for the sign patterns of $A_rx$ and $x$.
	 	
	\noindent \textbf{(2).} By
	 Theorem~\ref{Twhitney}, there exists a sequence $A_r^{(l)} \to A_r$ of $TP_{r-1}$
	 matrices.  Now
	 $A_r^{(l)}$ is $TP_{r-1}$, so by Theorem~\ref{thrmlast}(2) there exists $i_l
	 \in [r]$ such that $x_{i_l} (A_r^{(l)} x)_{i_l} > 0$. Hence there exists $i_0 \in [r]$ and an
	 increasing subsequence of positive integers $l_p$ such that $i_{l_p}=i_0$ for all $p \geq 1$. Now $(2)$ follows:
	 \[
	 x_{i_0} (A_rx)_{i_0} = \lim_{p \to \infty} x_{i_{l_p}} (A_r^{(l_p)}
	 x)_{i_{l_p}} \geq 0.
	 \]
	 	\end{proof}
By Example \ref{examplelast}, we have a $TN_{r-1}$ matrix which is not $TN_r$. This gives us the following proposition.
\begin{prop}
	Test-vectors from any open orthant apart from the open bi-orthant $\altr{r}$, can not be used to characterize total non-negativity via either variation diminution or sign non-reversal.
\end{prop}

We conclude with a similar observation about the LCP: Theorem \ref{tp-lcp} shows that for $TP_k$ matrices $A$, the solution sets to certain LCPs cannot simultaneously contain two vectors with alternately zero and positive coordinates (and disjoint supports). Our final result shows that if $A \in \mathbb{R}^{r \times r}$ is merely $TP_{r-1}$, the same holds when `alternating' is replaced by `not always alternating'. Thus, the `alternation' is also distinguished for the LCP-characterization of total positivity.

\begin{prop}
	If $A \in \mathbb{R}^{r \times r}$ with $r\geq 2$ is a $TP_{r-1}$ matrix, then $\sol{A_r}$ with $q<0$ does not simultaneously contain two vectors which have disjoint supports, and at least one of which has two consecutive positive coordinates.
\end{prop}

The proof is analogous to Theorem \ref{tp-lcp} using Theorem \ref{thrmlast} $(2)$.
Hence by Example \ref{examplelast}, totally positive matrices can not be identified by the solution sets $\sol{A}$ of LCP, which does not simultaneously contain two vectors with disjoint supports, and at least one of which has two consecutive positive coordinates.

\acknowledgments{
I thank the referee for carefully going through the paper and
for their suggestions. I also thank J\"{u}rgen Garloff and Apoorva Khare for a detailed reading of an earlier draft and for providing valuable feedback.
}

\affiliationone{
   Projesh Nath Choudhury\\
   Department of Mathematics\\
   Indian Institute of science\\
   Bangalore 560012\\India
   \email{projeshc@iisc.ac.in\\ projeshnc@alumni.iitm.ac.in
   }}


\begin{thebibliography}{9}
	
	
	\bibitem{BGKP20}
	A. Belton, D. Guillot, A. Khare, and M. Putinar.
	\newblock  Totally positive kernels, Polya frequency functions, and their transforms. {\em Prepint,}
	\newblock \href{https://arxiv.org/abs/2006.16213}{ 	arXiv:2006.16213}, 2020.
	
	\bibitem{BGKP21}
	A. Belton, D. Guillot, A. Khare, and M. Putinar.
	\newblock  Hirschman-Widder densities. {\em Prepint,}
	\newblock \href{https://arxiv.org/abs/2101.02129}{ 	arXiv:2101.02129}, 2021.
	
	\bibitem{BFZ96}
	A.~Berenstein, S.~Fomin, and A.~Zelevinsky.
	\newblock Parametrizations of canonical bases and totally positive matrices.
	\newblock \href{http://dx.doi.org/10.1006/aima.1996.0057}{\em Adv.
		Math.}, 122:49--149, 1996.
	
	\bibitem{deBoor}
	C.~de Boor.
	\newblock On calculating with $B$-splines.
	\newblock \href{http://dx.doi.org/10.1016/0021-9045(72)90080-9}{\em J.\
		Approx.\ Theory}, 6(1):50--62, 1972.
	
	\bibitem{Bre95}
	F.~Brenti.
	\newblock Combinatorics and total positivity.
	\newblock \href{http://dx.doi.org/10.1016/0097-3165(95)90000-4}%
	{\em J.\ Combin.\ Theory Ser.~A}, 71(2):175--218, 1995.
	
	\bibitem{BJM81}
	L.D.~Brown, I.M.~Johnstone, and K.B.~MacGibbon.
	\newblock Variation diminishing transformations: a direct approach to total positivity and its statistical applications.
	\newblock \href{https://doi.org/10.2307/2287577 }%
	{\em J. Amer. Statist. Assoc.}, 76(376):824--832, 1981.
	
	\bibitem{Chepuri}
	A.~Brosowsky, S.~Chepuri, and A.~Mason.
	\newblock Parametrizations of {$k$}-non-negative matrices: cluster algebras and
	{$k$}-positivity tests.
	\newblock \href{https://doi.org/10.1016/j.jcta.2020.105217}{\em J.
		Combin. Theory Ser. A}, 174:art.~105217, 25 pp., 2020.
	
	\bibitem{CKK21}
	P.N.~Choudhury, M.R.~Kannan, and A.~ Khare.
	\newblock Sign non-reversal property for totally non-negative and totally positive matrices, and testing total positivity of their interval hull.
	\newblock \href{http://dx.doi.org/10.1112/blms.12475}{\em Bull.\ London
		Math.\ Soc.}, 53(4):981--990, 2021.
	
	
	\bibitem{Co68}
	R.W.~Cottle.
	\newblock On a problem in linear inequalities.
	\newblock \href{https://doi.org/10.1112/jlms/s1-43.1.378}{\em J. London Math. Soc.}, 43:378--384, 1968.
	
	\bibitem{CD68}
	R.W.~Cottle and G.B.~Dantzig.
	\newblock Complementary pivot theory of mathematical programming.
	\newblock \href{https://doi.org/10.1016/0024-3795(68)90052-9}{\em Linear Algebra Appl.}, 1(1):103--125, 1968.
	
	\bibitem{CPS09}
	R.W.~Cottle, J-S~Pang, and R.E.~Stone.
	\newblock {\em The linear complementarity problem}.
	\newblock \href{ https://doi.org/10.1137/1.9780898719000}{Classics in Applied Mathematics}, SIAM, Philadelphia, PA, 2009.
	
	\bibitem{Cr71}
	C.W.~Cryer.
	\newblock The solution of a quadratic programming problem using systematic overrelaxation.
	\newblock \href{https://epubs.siam.org/doi/pdf/10.1137/0309028}{\em SIAM J. Control}, 9:385--392, 1971.
	
	\bibitem{Curry}
	H.B.~Curry and I.J.~Schoenberg.
	\newblock On P\'olya frequency functions IV: the fundamental spline
	functions and their limits.
	\newblock \href{http://dx.doi.org/10.1007/BF02788653}{\em J.\ d'Analyse
		Math.}, 17:71--107, 1966.
	
	\bibitem{ES76}
	B.C.~Eaves and H.~Scarf.
	\newblock The solution of systems of piecewise linear equations.
	\newblock \href{https://doi.org/10.1287/moor.1.1.1}{\em  Math. Oper. Res.}, 1(1):1--27, 1976.
	
	\bibitem{fallat-john}
	S.M. Fallat and C.R. Johnson.
	\newblock {\em Totally non-negative matrices}.
	\newblock
	\href{https://press.princeton.edu/books/hardcover/9780691121574/totally-non-negative-matrices}{Princeton
		Series in Applied Mathematics}, Princeton University Press, Princeton,
	2011.
	
	
	\bibitem{FP12}
	M.~Fekete and G.~P\'{o}lya.
	\newblock \"{U}ber ein {P}roblem von {L}aguerre.
	\newblock \href{http://dx.doi.org/10.1007/BF03015009}{\em Rend.\ Circ.\
		Mat.\ Palermo}, 34:89--120, 1912.
	
	\bibitem{FZ00}
	S.~Fomin and A.~Zelevinsky.
	\newblock Total positivity: tests and parametrizations.
	\newblock \href{http://dx.doi.org/10.1007/BF03024444}%
	{\em Math.\ Intelligencer}, 22(1):23--33, 2000.
	
	\bibitem{FZ02}
	S.~Fomin and A.~Zelevinsky.
	\newblock Cluster algebras. {I}. {F}oundations.
	\newblock \href{http://dx.doi.org/10.1090/S0894-0347-01-00385-X}{\em J.\
		Amer.\ Math.\ Soc.}, 15(2):497--529, 2002.
	
	\bibitem{gale-nikai-pmat}
	D.~Gale and H.~Nikaido.
	\newblock The {J}acobian matrix and global univalence of mappings.
	\newblock \href{http://dx.doi.org/10.1007/BF01360282}{\em Math.\ Ann.},
	159:81--93, 1965.
	
	\bibitem{gantmacher-krein}
	F.R. Gantmacher and M.G. Krein.
	\newblock Sur les matrices compl\`etement nonn\'egatives et oscillatoires.
	\newblock \href{http://www.numdam.org/item?id=CM_1937__4__445_0}%
	{\em Compositio Math.}, 4:445--476, 1937.
	
	\bibitem{GK50}
	F.R. Gantmacher and M.G. Krein.
	\newblock {\em Oscillyacionye matricy i yadra i malye kolebaniya
		mehani\v{c}eskih sistem}.
	\newblock Gosudarstv. Isdat. Tehn.-Teor. Lit., Moscow-Leningrad, 1950.
	\newblock 2d ed.
	
	\bibitem{GW96}
	J.~Garloff and D.G.~ Wagner.
	\newblock Hadamard products of stable polynomials are stable.
	\newblock \href{https://doi.org/10.1006/jmaa.1996.0348}{\em J. Math. Anal. Appl.}, 202(3):797--809, 1996.
	
	\bibitem{GRS18}
	K.~Gr\"{o}chenig, J.L. Romero, and J.~St\"{o}ckler.
	\newblock Sampling theorems for shift-invariant spaces, {G}abor frames, and
	totally positive functions.
	\newblock \href{http://dx.doi.org/10.1007/s00222-017-0760-2}{\em Invent.\
		Math.}, 211:1119--1148, 2018.
	
	\bibitem{I66}
	A.W. Ingleton.
	\newblock A problem in linear inequalities.
	\newblock \href{https://doi.org/10.1112/plms/s3-16.1.519}{\em Proc. London Math. Soc. (3)}, 16:519--536, 1966.
	
	\bibitem{I70}
	A.W. Ingleton.
	\newblock The linear complementarity problem.
	\newblock \href{https://doi.org/10.1112/jlms/s2-2.2.330}{\em J. London Math. Soc. (2)}, 2:330--336, 1970.
	
	\bibitem{Karlin64}
	S.~Karlin.
	\newblock Total positivity, absorption probabilities and applications.
	\newblock \href{https://doi.org/10.2307/1993667}{\em Trans. Amer. Math. Soc.}, 111:33--107, 1964.
	
	\bibitem{K68}
	S.~Karlin.
	\newblock {\em Total positivity. {V}ol. {I}}.
	\newblock Stanford University Press, Stanford, CA, 1968.
	
	\bibitem{Karlinsplines}
	S.~Karlin and Z.~Ziegler.
	\newblock Chebyshevian spline functions.
	\newblock \href{http://dx.doi.org/10.1137/0703044}{\em SIAM J.\ Numer.\
		Anal.}, 3(3):514--543, 1966.
	
	
	
	\bibitem{Khare20}
	A. Khare.
	\newblock  Critical exponents for total positivity, individual kernel encoders, and the Jain-Karlin-Schoenberg kernel. {\em Prepint,}
	\newblock \href{https://arxiv.org/abs/2008.05121v1}{ 	 	arXiv:2008.05121}, 2020.
	
	\bibitem{KW14}
	Y.~Kodama and L.~Williams.
	\newblock K{P} solitons and total positivity for the {G}rassmannian.
	\newblock \href{http://dx.doi.org/10.1007/s00222-014-0506-3}%
	{\em Invent.\ Math.}, 198(3):637--699, 2014.
	
	\bibitem{Laguerre}
	E.~Laguerre.
	\newblock M\'emoire sur la th\'eorie des \'equations num\'eriques.
	\href{http://sites.mathdoc.fr/JMPA/PDF/JMPA_1883_3_9_A5_0.pdf}{\em J.\
		Math.\ Pures Appl.}, 9:9--146, 1883.
	
	\bibitem{L65}
	C.E.~Lemke.
	\newblock Bimatrix equilibrium points and mathematical programming.
	\newblock \href{https://doi.org/10.1016/0024-3795(68)90052-9}{\em Management Science}, 11(7):681--689, 1965.
	
	\bibitem{Lo55}
	C.~Loewner.
	\newblock On totally positive matrices.
	\newblock \href{https://doi.org/10.1007/BF01187945}{\em Math. Z.}, 63:338--340, 1955.
	
	\bibitem{Lu94}
	G.~Lusztig.
	\newblock Total positivity in reductive groups.
	\newblock In {\em Lie theory and geometry}, volume 123 of {\em Progr. Math.},
	pages 531--568. Birkh\"{a}user, Boston, MA, 1994.
	
	\bibitem{MS19}
	M.~Margaliot and E.D.~Sontag.
	\newblock Revisiting totally positive differential systems: a tutorial and new results.
	\newblock \href{https://doi.org/10.1016/j.automatica.2018.11.016}
	{\em Automatica J. IFAC}, 101:1--14, 2019.
	
	
	\bibitem{pinkus}
	A.~Pinkus.
	\newblock {\em Totally positive matrices}.
	\newblock \href{http://dx.doi.org/10.1017/CBO9780511691713}{Cambridge
		Tracts in Mathematics, Vol.~181}, Cambridge University Press, Cambridge,
	2010.
	
	\bibitem{Ri03}
	K.C. Rietsch.
	\newblock Totally positive {T}oeplitz matrices and quantum cohomology of
	partial flag varieties.
	\newblock \href{http://dx.doi.org/10.1090/S0894-0347-02-00412-5}%
	{\em J.\ Amer.\ Math.\ Soc.}, 16(2):363--392, 2003.
	
	\bibitem{STW58}
	H.~Samelson, R.M.~Thrall, and O.~Wesler.
	\newblock A partition theorem for {E}uclidean {$n$}-space.
	\newblock \href{https://doi.org/10.2307/2033091}%
	{\em Proc. Amer. Math. Soc.}, 9:805--807, 1958.
	
	
	
	\bibitem{S30}
	I.J. Schoenberg.
	\newblock \"{U}ber variationsvermindernde lineare {T}ransformationen.
	\newblock \href{http://dx.doi.org/10.1007/BF01194637}{\em Math.\ Z.},
	32:321--328, 1930.
	
	\bibitem{Schoenberg46}
	I.J. Schoenberg.
	\newblock Contributions to the problem of approximation of equidistant
	data by analytic functions. Part A. On the problem of smoothing or
	graduation. A first class of analytic approximation formulae.
	\newblock \href{http://dx.doi.org/10.1090/qam/15914}{\em Quart.\ Appl.\
		Math.}, 4(1):45--99, 1946.
	
	\bibitem{S55}
	I.J. Schoenberg.
	\newblock On the zeros of the generating functions of multiply positive
	sequences and functions.
	\newblock \href{http://dx.doi.org/10.2307/1970073}{\em Ann.\ of Math.\
		(2)}, 62(3):447--471, 1955.
	
	
	\bibitem{SW53}
	I.J. Schoenberg and A.M. Whitney.
	\newblock On P\'olya frequency functions. III. The positivity of translation determinants with an application to the interpolation problem by spline curves.
	\newblock \href{https://doi.org/10.2307/1990881 }{\em Trans. Amer. Math. Soc.}, 74:246--259, 1953.
	
	\bibitem{Whitney}
	A.M. Whitney.
	\newblock A reduction theorem for totally positive matrices.
	\newblock \href{http://dx.doi.org/10.1007/BF02786969}%
	{\em J.\ d'Analyse Math.}, 2(1):88--92, 1952.
	
\end{thebibliography}
\end{document}